\documentclass[10pt]{amsart}

\usepackage[margin=1in]{geometry}
\usepackage{paralist}
				
\usepackage{amsthm}

\newtheorem{theorem}{Theorem}
\newtheorem{conjecture}[theorem]{Conjecture}
\newtheorem{lemma}[theorem]{Lemma}
\newtheorem{lemma*}{Lemma*}
\newtheorem{prop}[theorem]{Proposition}

\theoremstyle{definition}
\newtheorem{question}{Question}
\newtheorem{remark}[theorem]{Remark}

\theoremstyle{definition}
\newtheorem{example}{Example}
\newtheorem{example*}[example]{Example*}

\renewcommand{\P}{\mathbb P}
\newcommand{\R}{\mathbb R}
\newcommand{\Z}{\mathbb Z}

\DeclareMathOperator{\vdim}{vdim}

\newcommand{\cO}{\mathcal O}

\DeclareMathOperator{\edim}{\rm{expdim}}

\begin{document}

\title{A few questions about curves on surfaces}

\author[Ciliberto]{Ciro Ciliberto}
\address{Universit\`a degli Studi di Roma Tor Vergata, Via della
  Ricerca Scientifica 1, 00133, Rome,
  Italy}\email{cilibert@axp.mat.uniroma2.it}
  
  \author[Knutsen]{Andreas Leopold Knutsen}
\address{Department of Mathematics, University of Bergen, Postboks 7800,
5020 Bergen, Norway}
\email{andreas.knutsen@math.uib.no}

\author[Lesieutre]{John Lesieutre}
\address{University of Illinois at Chicago,
Math, Stat \& CS, Room 411 SEO,
851 S Morgan St, M/C 249,
Chicago, IL 60607-7045, USA}
\email{jdl@uic.edu}

 \author[Lozovanu]{Victor Lozovanu}
 \address{Universit\`a degli Studi di Milano--Bicocca, Via Cozzi, 55, 20126, Milano, Italy}
\email{victor.lozovanu@unimib.it}
 \email{victor.lozovanu@gmail.com}
  
  \author[Miranda]{Rick Miranda}
\address{Colorado State University, Department of mathematics, College of Natural Sciences,
117 Statistics Building, Fort Collins, CO 80523, USA }
\email{rick.miranda@math.colostate.edu}

 \author[Mustopa]{Yusuf Mustopa}
\address{Tufts University,
Department of Mathematics,
503 Boston Avenue,
Bromfield-Pearson
Room 106,
Medford, MA 02155, USA}
\email{yusuf.mustopa@tufts.edu}

\author[Testa]{Damiano Testa}
\address{Mathematics Institute, Zeeman Building, University of Warwick, Coventry CV4 7AL, Great Britain}
\email{adomani@gmail.com}


\begin{abstract} In this note we address the following kind of question: let $X$ be a smooth, irreducible, projective surface and \(D\) a divisor on $X$ satisfying some sort
of positivity hypothesis, then is there some  multiple of $D$ depending only on $X$  which is effective or movable?  We describe some
examples, discuss some conjectures and prove some results that suggest that the answer should in general be negative, unless one puts some really strong hypotheses either on \(D\) or on \(X\).
\end{abstract}
\maketitle

\section{Introduction}


Let \(X\) be a smooth, irreducible, projective, complex surface, henceforth simply called surface.  A celebrated result
by Franchetta and Bombieri (see~\cite{B,F1,F2}) implies that if $X$ is of general type, i.e., the canonical line bundle \(K_X\) is big, then \(h^0(X,\cO_X(5K_X))
\geq 3\).  It seems natural to ask whether there might be similar
effective non--vanishing results for divisors on \(X\) other than the
canonical, asserting that if a divisor \(D\) satisfies some type
of positivity hypothesis, then some multiple $mD$ of that divisor, with $m$ depending only on $X$,
must be \emph{effective}, i.e.,  $h^ 0(X, \cO_X(mD))>0$, or even \emph{movable}, i.e.\ $h^ 0(X, \cO_X(mD))>1$.  

It appears that this is unlikely
without some strong hypotheses either on \(D\) or on \(X\). Here we describe some
examples, discuss some conjectures and prove some results that any future research in this direction should take into
account.  We have been motivated by the following three questions, asked by A. L. Knutsen during the Warsaw workshop ``Okounkov bodies and Nagata type Conjectures'' held at the Banach Centre in September 2013,  discussed there and also during the workshop ``Recent advances in linear series and Newton--Okounkov bodies'' held in Padua in February 2015.

\begin{question}
\label{q1}
  Does there exist a constant \(m_1=m_1(X)\) such that if \(D\) is any
  divisor with \(h^0(X,\cO_X(D)) = 1\) and $D^2>0$, then one has  \(h^0(X,\cO_X(m_1D)) \geq 2\)?
\end{question}
\begin{question}
\label{q2}
  Does there exist a constant \(m_2=m_2(X)\) such that if \(D\) is a
  divisor with \(D^2 > 0\) and $H\cdot D>0$ for an ample divisor $H$,  then one has \(h^0(X,\cO_X(m_2 D)) > 0\)?
\end{question}
\begin{question}
\label{q3}
 Does there exist a constant \(m_3=m_3(X)\) such that if \(D\) is a
  divisor with \(D^2 \geq m_3\) and $H\cdot D>0$ for an ample divisor $H$, then one has \(h^0(X,\cO_X(D))
  > 0 \)?
\end{question}

Blow-ups of \(\P^2\) at \(r \geq 10\) very general points, having quite complicated nef cones, appear to
provide a fertile source of counterexamples, though in most cases we
will find it necessary to assume the Segre--Harbourne--Gimigliano--Hirschowitz (SHGH) conjecture (see~\S\ref{Sec:SHGH}) in order to
compute the dimensions of various linear systems. Under the
hypothesis that SHGH conjecture holds, we show that Questions~\ref{q2} and~\ref{q3} both have
negative answers (see~\S\ref{Sec:Pell}).  To the other extreme, Question~\ref{q2} has an affirmative answer if  the nef cone of $X$ is as simple as possible, i.e., it is rational polyhedral. This is shown in~\S\ref{sec:rp}. 

We have not been able to answer Question~\ref{q1}. In~\S\ref{sec:q1} we relate it to the following:

\begin{question}\label{q4}
  Does there exist a constant \(m_4=m_4(X)\) such
  that \[ \frac{ D \cdot K_X}{D^ 2}  < m_4\] for any effective
  divisor \(D\) on \(X\) with $D^ 2>0$?
\end{question}

If one can answer Question~\ref{q4} affirmatively for a surface $X$, then the same happens for Question~\ref{q1}. However the answer to Question~\ref{q4} is in general negative, so the problem arises to classify surfaces $X$ for which  Question~\ref{q4} has an affirmative answer.

Finally in~\S\ref{Sec:Harb} we discuss a conjecture by B. Harbourne which  also sits in this same circle of ideas. 

\subsection*{Conventions} We use standard notation and terminology. We will often use the same notation for divisors, divisors classes and line bundles, hoping that no confusion arises.  

\section{The Segre--Harbourne--Gimigliano--Hirschowitz conjecture}\label{Sec:SHGH}

  The \emph{virtual dimension} of the complete linear system  \(|L|\) on a smooth projective
  surface \(X\) is 
  \[\vdim(L) = \chi(L) - 1 =p_a(X)+ \frac{L \cdot
    (L-K_S)}{2}.\]  
    
 Let $f \colon X_r\to \mathbb P^ 2$ be the blow--up of the plane at $r$ \emph{very general} points, with exceptional divisors $E_1,\ldots, E_r$. 
We set $E=E_1+\cdots+ E_r$ and denote by $H$ the total transform of a line via $f$.
If \[D = dH -
 \sum_{i=1}^r m_i E_i=:(d; m_1,\ldots,m_r),\] then
\[
\vdim(D) = \frac{d(d+3)}{2} - \sum_{i=1}^r \frac{m_i(m_i+1)}{2}=\frac {D^ 2-K_{X_r}\cdot D}2
\]
and the \emph{expected dimension} of the linear system $|D|$ is
\[
\edim(D)=\max \{-1,\vdim(D)\}.
\]

We will assume $m_1\ge \ldots\ge m_r>0$ and we will use the SHGH conjecture in the following form due to Gimigliano (see~\cite{G}).

\begin{conjecture}[SHGH Conjecture]
  Suppose that \(D =(d;m_1,\ldots, m_r)\), with \(d > m_1 + m_2 + m_3\).
  Then \(\dim (D) = \edim (D)\).
\end{conjecture}

In what follows, we will mainly consider \emph{homogeneous} divisor classes on $X_r$, namely \[D = dH -mE =: (d;m^r).\]  

\section{Pell divisors and counterexamples to Questions~\ref{q2} and~\ref{q3} }\label{Sec:Pell}

Let \(X\) be the blow-up of \(\P^2\) at \(r\) very general points and consider homogeneous divisor classes $D=(d;m^ r)$, such that \(\vdim(D) = 0\). This condition boils down to the 
Pell-type equation
\begin{equation}
x^2 - ry^2 = 9-r. 
\end{equation}
where  \(x = 2d+3\) and \(y=2m+1\). 

In the case \(r=10\), the solution is particularly simple (the
case $r > 10$ not a square can be treated in a similar way, but we do not dwell on this here). 
Let $C_k=\frac {p_k}{q_k}$ be the \emph{$k$--th convergent} 
of the simple continued fraction of $\sqrt{10}=[3;\bar 6]$, with $(p_k,q_k)=1$. One has
$(p_0,q_0)=(3,1)$ and 
\begin{equation}\label{eq:ric}
\left(3 + \sqrt{10} \right)^{k+1}=p_k+q_k \sqrt{10} \text{, for all}\,\,\, k\in \mathbb N.
\end{equation}
The norm of  $3 + \sqrt{10}$ in $\mathbb Z[\sqrt{10}]$ being $-1$, the solutions of the equation
\[
x^2 - 10y^2 = -1 \quad [\text{resp.\ of}\,\,\,  x^2 - 10y^2 = 1]
\]
are 
\[
 x=p_{2k}, \,\, y=q_{2k}, \quad [\text{resp.}\,\,\,  x=p_{2k+1}, \,\, y=q_{2k+1}], \,\,\, \text{for all}\,\,\, k\in \mathbb N.
\]
It is easy to verify, using~\eqref{eq:ric}, that every such solution has both \(x\) and \(y\) odd.

Hence we have a 
sequence of divisor classes $D_k=(d_k; m_k^ {10})$ on $X_{10}$ with
\[
 d_k=\frac {p_{2k}-3}2 , \; m_k=\frac {q_{2k}-1}2 \text{, for all } k \in \mathbb N \text{, and } \vdim (D_k)=0,
\]
which we call \emph{Pell divisors}.  The properties of continued fractions  imply that $C_{2k}>\sqrt {10}>3$, which yields 
\[ d_k>3m_k,\,\,\, \text{for all}\,\,\,   k\in \mathbb N.\] 
Thus,  \emph{if SHGH conjecture holds}, we have 
\[ \dim(|D_k|) = 0  \,\,\, \text{for all}\,\,\, k\in \mathbb N. \]

\begin{lemma}[Divisibility Lemma]
  For  all positive $k\in \mathbb N$, one has 
\[ D_k= c_k F_k,\] where
  \begin{equation*}\label{eq:fn}
\begin{cases}
c_k = p_{k-1} \\
F_k = (p_k;q_k^ {10})
\end{cases}
\text{for \(k\) odd}, \qquad 
\begin{cases}
c_k = q_{k-1} \\
F_k = (10q_k;p_k^ {10})
\end{cases}
\text{for \(k\) even.}
\end{equation*}
\end{lemma}
\begin{proof}
It may be verified with straightforward calculation.
\end{proof}

The coefficients can be calculated directly using the recurrence
\begin{align*}
d_{k+1} = 19d_k + 60m_k + 57 \\
m_{k+1} = 6d_k + 19m_k + 18,
\end{align*}
with \(d_0 = m_0 = 0\) and $k\in \mathbb N-\{0\}$. The first few such values are listed below.
\[
\begin{array}{|r|r|r|r|r|} 
\hline
k & (p_k,q_k) & D_k & c_k & F_k \\
\hline \hline
0 & (3,1) & (0;0^ {10}) & & \\\hline
1 & (19,6) & (57;18^ {10}) & 3 & (19;6^ {10}) \\\hline
2 & (117,37) & (2220;702^ {10}) & 6 & (370;117^ {10}) \\\hline
3 & (721,228) & (84357;26676^ {10}) & 117 & (721;228^ {10}) \\\hline
4 & (4443,1405) & (3203400;1013004^ {10}) & 228 & (14050;4443^ {10}) 
\\
\hline
\end{array}
\]

\begin{lemma}\label{lem:vdim}
  For every \(h <
  c_k\), one has \(\vdim(hF_k) < 0\).
\end{lemma}
\begin{proof}
For odd \(k\), by taking into account~\eqref{eq:ric},  one has
\begin{align*}
  \vdim(h F_k) &= \frac{(h p_k)(h p_k + 3)}{2} - 10 \frac{(h q_k)(h q_k+1)}{2} \\ 
&= \frac{h^2}{2} \left( p_k^2 - 10q_k^2 \right) + \frac{h}{2} \left( 3p_k - 10q_k \right) \\
&= \frac{h}{2}(h-p_{k-1}).
\end{align*}
This is \(0\) in the case \(h = c_k = p_{k-1}\), and negative if \(h <
p_{k-1}\).
The calculation for even \(n\) is analogous.
\end{proof}

\begin{prop}\label{prop:q2}  If the SHGH conjecture holds, then
  the divisors \(F_k\), for \(k\) any odd positive integer, provide a negative answer to
  Question~\ref{q2}.
\end{prop}
\begin{proof}
  If \(k\) is odd, then \(F_k^2 = p_k^2 - 10q_k^2 = 1\).  Moreover, by Lemma~\ref{lem:vdim}, one has  \(\vdim(hF_k) < 0\) for \(0 \leq h < c_k\) and so $|hF_k|$ is empty.  On the
  other hand, \(\vdim(c_k F_k) = 0\) and so \(|D_k|\) is
  effective. As 
$c_k$ attains arbitrarily large values, this provides a negative answer to Question~\ref{q2}.
\end{proof}

\begin{remark}
  The \(k=1\) case is not subject to SHGH conjecture: by~\cite{CM}, the divisor \((57;18^{10})\) is effective, but 
  \((19;6^{10})\) and \((38;12^{10})\) are not.
\end{remark}

\begin{prop}\label{prop:q3} If SHGH conjecture holds, then
the divisors \((c_k-1) F_k\), for \(k\) any odd positive integer,  give a counterexample to Question~\ref{q3}.
\end{prop}
\begin{proof}
  The same calculation as above shows that \(((c_k-1)F_k)^2 =
  (c_k-1)^2\), which can be arbitrarily large, but
  \(\vdim(X,\cO_X((c_k-1)F_k)) < 0\).
\end{proof}

\begin{remark}
  Whenever we use the SHGH conjecture in this section, we really use it for 10 points. Thus, in Propositions~\ref{prop:q2} and~\ref{prop:q3} 
we could assume that the SHGH conjecture holds for 10 points. \end{remark}

\section{Question~\ref{q2}  when the nef cone is rational polyhedral}\label{sec:rp}

In this section, for a smooth, irreducible, projective surface $X$ we use the (by now standard) notation from~\cite{L}: for instance, ${\rm NS}(X):=\textup{N}^1(X)$ is the {\emph{N\'eron--Severi group}}, ${\rm N}^ 1(X)_{\mathbb R} := \textup{N}^1(X) \otimes \mathbb{R}$ is the \emph{N\'eron--Severi space}, ${\rm Nef}(X)$ is the {\emph{nef cone}}, ${\rm Big}(X)$ is the {\emph{big cone}}, $\overline {\rm{NE}}(X)$ is the {\emph{pseudo--effective cone}}, and so on.

In the paper~\cite{KL}, the authors prove the following result.

\begin{theorem} {\emph{(\cite[Theorem~4.10]{KL})}}.
Let $X$ be a smooth, irreducible, projective surface, let $D$ be a divisor on $X$ and assume that 
the \emph{nef cone} (or, dually, the \emph{pseudo--effective cone}) of the surface $X$ is rational polyhedral.
Question~\ref{q2} restricted to {\emph{ample}} divisors $D$ on $X$ admits an affirmative answer.
\end{theorem}
The proof is based on Fujita's vanishing theorem and diophatine approximation.


Following the ideas from~\cite{KL} and using a local version of a famous theorem of Anghern--Siu about the generation of adjoint line bundles as in~\cite{ELMNP}, in this section we extend~\cite[Theorem~4.10]{KL}.  Indeed, we give an affirmative answer to Question~\ref{q2}, whenever the nef cone of the surface is rational polyhedral and $D$ is big (in particular, our result applies without restriction to the divisors appearing in Question~\ref{q2}).

Before stating and proving this result, we recall some notation. If $D$ is a pseudoeffective divisor on $X$, one has the \emph{Zariski decomposition} $D=P(D) + N(D)$, where $P(D)$ is the \emph{nef part} of $D$ and $N(D)$ is the \emph {negative part} of $D$. If $D$ is a big divisor on $X$, one defines
${\rm Null}(D)$ to be the divisor (containing $N(D)$) given by the sum of all irreducible curves $E$ on $X$ such that ${\rm P}(D)\cdot E=0$. 

\begin{theorem}\label{thm:rational}
Let $X$ be a smooth, irreducible, projective surface with rational polyhedral nef cone (or pseudo--effective cone). Then there exists an integer $m:=m(X)>0$ such that for any big divisor $D$ on $X$ one has $h^0(X, \mathcal O_X(mD))>0$.
\end{theorem} 

\begin{remark}
Theorem~\ref{thm:rational} explains the need we had in~\S\ref{Sec:Pell} to take a sequence of divisor classes whose limit is an irrational class in ${\rm N}^ 1(X)_{\mathbb R}$. In order to find surfaces $X$ which are counterexamples to an affirmative answer to Question~\ref{q2} one needs the pseudo--effective cone of $X$ to be complicated, and it is well known that the blow--up of the projective plane at $10$ or more very general points is such that $\textup{Nef}(X)$ is far from being rational polyhedral. 
\end{remark}

\begin{proof}[Proof of Theorem~\ref{thm:rational}]
The proof consists in two steps:
\begin{itemize}
\item
first, we show that there exists a translate of the big cone in $\textup{N}^1(X)_{\R}$ such that any divisor class simultaneously in this translate and in the N\'eron--Severi group is effective;
\item
second, we use diophantine approximation to deduce the theorem.
\end{itemize}

\noindent
\textit{Step 1:} \emph {there exists an ample divisor $R\in \textup{N}^1(X)$, such that}
\begin{equation}\label{eq:transl}
\forall \ D\in \big(R+\textup{Big}(X)\big) \ \bigcap \ \textup{N}^1(X) \quad {\textrm{it follows that}} \quad h^0(X, \cO_X(D))>0 \ .
\end{equation}

The condition that the nef cone is rational polyhedral implies that there exist only finitely many \emph{negative curves} on $X$, i.e., irreducible curves $E$ such that $E^ 2<0$; denote by $E_1,\ldots,E_h$ these negative curves.  The negative part of any pseudoeffective divisor on $X$ is of the form
$\sum_{i=1}^ ha_iE_i$, with rational numbers $a_i\geq 0$, for $i\in \{1,\ldots, h\}$. 

 We choose a point $x\in X$ that does not sit on $E_1 \cup \cdots \cup E_h$ and we choose an ample divisor $A\in \textup{N}^1(X)$ such that:\\ \begin{inparaenum} \item [(a)] $A^2>9$;\\
 \item [(b)] $A\cdot C>3$ for any irreducible curve $C\subset X$ passing through the point $x$;\\
 \item [(c)] the adjoint line bundle $K_X+A$ is ample.\end{inparaenum} 
 
We will prove that $R:= K_X+A$ verifies~\eqref{eq:transl}.
In order to do so, we will use~\cite[Theorem~2.20]{ELMNP}, which is a generalization to the big case of a theorem of Anghern and Siu. In the case of surfaces, it says that for a big divisor $B\in \textup{N}^1(X)$, and a point $x\notin \textup{Null}(B)$ such that 
\begin{equation}\label{eq:as}
\textup{vol}_X(B) > 9 \ \textup{ and } \ \textup{vol}_C(B) > 3 \ \textup{ for any curve } \ C\subset X \ \textup{ passing through } x ,
\end{equation}
then $x$ is not in the base locus of $K_X+B$; in particular, the divisor $K_X+B$ is effective.

In the case of surfaces the volumes appearing in~\eqref{eq:as} are easy to compute from the Zariski decomposition $B=P(B)+N(B)$, as explained in~\cite[Example~2.19]{ELMNP}. If  $C$ is not contained in $N(B)$ (which is the case, since $x\in C$ and $x\notin N(B)$), then 
\[
\textup{vol}_X(B)=\textup{vol}_X(P(B)) \quad \textup{ and } \quad \textup{vol}_C(B)=\textup{vol}_C(P(B)) .
\]

Going back to our setup, let $D$ be any big divisor on $X$ and set $B=A+D$. We prove that~\eqref{eq:as} holds, so that we may apply to $B$ the aforementioned~\cite[Theorem~2.20]{ELMNP}.
Indeed,  $N(D)-N(B)=P(B)-A-P(D)$ is effective (see~\cite[Lemmas~14.8 and~14.10]{Ba}). Hence
\[
\textup{vol}_X(B)=\textup{vol}_X(P(B))\geq \textup{vol}_X(A+P(D)) =(A+P(D))^2>9 , 
\]
proving the first part of~\eqref{eq:as}. As for the second part, note that $N(D)-N(B)=P(B)-A-P(D)$ consists of negative curves, thus not passing through  $x$. Hence, 
no irreducible curve $C$ passing through $x$ is contained in $P(B)-A-P(D)$, thus 
\[
\textup{vol}_C(B)=\textup{vol}_C(P(B))\geq \textup{vol}_C(A+P(D)) =(A+P(D))\cdot C>3 .
\]
In conclusion, $K_X+B=K_X+A+D$ is effective for any big divisor $D$ on $X$, thus proving~\eqref{eq:transl} with $R=K_X+A$. \medskip

Before turning to Step 2, we notice that, by substituting $R$ with its sum with  an ample divisor one has 
\begin{equation}\label{eq:nn}
\forall \ D\in \big(R+\overline {\rm{NE}}(X)\big) \ \bigcap \ \textup{N}^1(X) \quad {\textrm{it follows that}} \quad h^0(X, \cO_X(D))>0 .
\end{equation}

\noindent
\textit{Step 2.} By taking~\eqref{eq:nn} into account, in order to accomplish
 the proof it suffices to apply to ${\rm{NE}}(X)$  the following property of rational polyhedral convex cones:  \emph{Let $\mathcal{C}\subset \R^n$ be a rational polyhedral convex cone and let $R\in \textup{int}(\mathcal{C})\cap \Z^n$; then there exists a natural number $m>0$ such that}
\[
\forall \ \xi\in \textup{int}(\mathcal{C}) \ \bigcap \ \Z^n \quad {\textrm{it follows that}} \quad m\xi\in R+\mathcal{C} .
\]

This result is shown in~\cite[Theorem~4.10]{KL}.  For the benefit of the reader, we include here the proof: this is the step where diophantine approximation comes into play.

Let $H\subset \R^n$ be an integral hyperplane through the origin, i.e., there exists a 
$u\in\Z^n$ such that $H=\{x\in\R^{n}| \langle x,u \rangle = 0\}$. The distance of $P\in \R^n$ from  $H$ is
\[
\textup{distance}(P,H) = \frac{|\langle P,u\rangle|}{||u||}\ .
\]
If  $P\in\Z^n$ and $P\notin H$, then $|\langle P,u\rangle |\geq 1$, hence $\textup{distance}(P,H) \geq  1/||u||$. 
Therefore, there is a constant $c>0$ such that $\textup{distance}(P,H)\geq c$ for any integral point $P\notin H$.

Since $\mathcal{C} \subset \R^n$ is rational polyhedral, it follows that the supporting hyperplanes of each face are  integral. Thus there exists a constant $c>0$ such that 
\begin{equation}\label{eq:diof}
\textup{distance}(P,\partial \mathcal{C}) \geq c \textup{, for any } P \in \textup{int} (\mathcal{C}) \cap \Z^n ,
\end{equation}
where $\partial\mathcal{C}$ denotes  the boundary of  $\mathcal{C}$ in $\R^n$.

Pick  $P\in\textup{int}(\mathcal{C})$, and let $\Lambda$ be the plane determined  by the origin $O\in \R^n$, $R$ and $P$. 
Let $\mathcal{C}_{\Lambda}=\mathcal{C}\cap \Lambda$. This is a cone in $\R^2$ with edges two half lines $\ell_i=\R_+v_i$  generated by the vectors $v_i$, for $i=1,2$.  Since  $\partial\mathcal{C}$ is supported by rational hyperplanes
and  $\Lambda$ is also defined by equations with rational coefficients, we may assume that $v_1, v_2$ are rational.

Furthermore,  $(R+\mathcal{C})\cap \Lambda$ is the cone $\ell_1+\ell_2$ translated by $R$. Without loss of generality, we may suppose that 
the half line $\R_+(OP)$ intersects first the half line $R+\ell_1$ at a point $D$. Then it suffices to find a constant $C>0$,
 not depending on $P$, such that $\frac{||OD||}{||OP||} < C$. Using similar triangles, one has
\[
\frac{||OD||}{||OP||} = \frac{\textup{distance}(D,\ell_1)}{\textup{distance}(P,\ell_1)} \ \leq \ \frac{||OR||}{c} ,
\]
where the latter inequality follows from~\eqref{eq:diof}. This finishes the proof of the theorem.
\end{proof}

\section{On Questions~\ref{q1} and~\ref{q4}}\label{sec:q1}

We have  not been able to find counterexamples to Question~\ref{q1}.  
As for Question~\ref{q4}, we note that:

\begin{lemma}\label{lem:impl} Given a surface $X$, if Question~\ref{q4}  is answered affirmatively,  then the same holds for Question~\ref{q1}.
\end{lemma}

\begin{proof} Assume Question~\ref{q4} is answered affirmatively.  First of all, we can  find a positive integer $N$ such that for any effective divisor $D$ on $X$ and for any integer $m>N$, the divisor  $K_X-mD$ is not effective: if $K_X$ is not effective, one takes $N=1$, otherwise one takes $N=H\cdot K_X$, with $H$ an ample divisor on $X$. 

Let $D$ be any effective divisor such that $D^ 2>0$. For all integers $m>N$, one has, by Riemann--Roch 
\[
\dim (|mD|)\ge \vdim (mD)\ge p_a(X)+\frac {m^ 2D^ 2-mD\cdot K}2>p_a(X)+\frac {m(m-m_4)}2.
\]
We can  certainly find an $m_1>N$ such that $m(m-m_4)$  
is large enough as soon as $m>m_1$,
thus answering  Question~\ref{q1} affirmatively. \end{proof}

However, Question~\ref{q4} has a negative answer in general.

\begin{example}\label{ex:K}(Attributed to J. Koll\'ar, see~\cite[Example~1.5.7] {L}).
Let \(E\) be an elliptic curve  and
  set \(Y = E \times E\), with \(F_1\) and \(F_2\) the numerical divisor classes
  of the fibers of the two projections to $E$.
  
Let $a,b$ be coprime integers and let $f_{a,b} \colon Y \to E$ be the morphism sending 
$(x,y)$ to $ax+by$ (where $+$ is the addition on $E$).  
Let $E_{a,b}\cong E$ be the general fibre of $f_{a,b}$.  Then
\[ F_1\cdot E_{a,b}= b^ 2,\,\,\, F_2\cdot E_{a,b}= a^ 2, \,\,\, E_{a,b}^ 2=0.\]
For a fixed non--zero integer $b$, set 
\[
A_n=F_1+E_{n,b},\,\,\, \text{so}\,\,\, A_n^ 2=2b^ 2>0,\,\,\, A_n\cdot (F_1 + F_2)=n^ 2+b^ 2+1 .
\] 

Pick $B\in |2(F_1 + F_2)|$ general and let \(f  \colon X \to Y\) be the double cover
  of \(Y\) branched over \(B\).  Take 
  \[
  D_n = f^*(A_n),\,\,\, \text{so that}\,\,\, D_n^ 2 = 4b^ 2.
  \]  
 However, the canonical divisor \(K_X\) is numerically equivalent to \(f^\ast (F_1+F_2)\), hence
  \[
  D_n \cdot K_X =  2 A_n \cdot (F_1 + F_2)=2(n^ 2+b^ 2+1),\,\,\, \text{so that}\,\,\, \lim_n\frac {D_n\cdot K_X}{D_n^ 2}=+\infty.
  \]
 \end{example}

A similar example can be constructed on a rational
surface.

\begin{example}
 Set \(Y:=X_r\), the blow-up of $\mathbb{P}^2$ at $r\geq 9$ very general points.
  Let \(A_0:=H\) be the strict transform of a general line in \(\P^2\).  By choosing elements from the
  Cremona orbit of \(A_0\), we may find a sequence of rational curves \(A_n\)
  for which \(A_n^2 = 1\), and \(A_n \cdot H \sim  n^m\) for any
  exponent \(m>0\).

Let \(f  \colon X \to Y\) be the double cover branched along the transform on $Y$ of a general conic of $\P^ 2$.
  We have \(K_X = f^\ast (K_Y + H)\).  Set \(D_n = f^\ast (A_n)\).  Then 
 \[D_n\cdot K_X \sim   n^ m, \,\,\, \text{and}\,\,\, D_n^ 2=2,\,\,\, \text{so that}\,\,\, \lim_n\frac {D_n\cdot K_X}{D_n^ 2}=+\infty.\]
 
 Note that  \(Y\) is isomorphic
  to the blow-up of \(\P^1 \times \P^1\) at \(2r\geq 18\) points in special
  position, or equivalently, to the blow-up of \(\P^2\) at \(2r+1\geq 19\) 
  points in special position. 
\end{example}

Of course the natural question arises:

 \begin{question} For which surfaces $X$ does Question~\ref{q4} have an affirmative answer?
  \end{question}
  
  \section{A conjecture by B. Harbourne}\label{Sec:Harb}
  
  The following conjecture is due to B. Harbourne.

\begin{conjecture}[See~\cite{AAVV}, Conjecture~2.5.3]\label {conj:H}
  Let \(X\) be a smooth projective surface.  There exists \(\alpha =
  \alpha(X)\) such that for every irreducible curve $D$ on $X$ one has  \(h^1(X,\mathcal O_X(D)) \leq \alpha h^0(X,\mathcal O_X(D))\).
\end{conjecture}

This conjecture has a different flavor according to the sign of $D^ 2$. If $D^ 2<0$, an affirmative answer to this question provides an affirmative answer to the famous Bounded Negativity Conjecture (see, e.g.~\cite{AAVV2}):

\begin{conjecture}[Bounded Negativity Conjecture]\label {conj:BNG}
  Let \(X\) be a smooth projective surface.  Then there is an integer $N=N(X)$ such that for any irreducible curve $D$ on $X$ one has $D^2>N$. 
\end{conjecture}

\begin{prop}
If Conjecture~\ref{conj:H} holds for a surface $X$ then the Bounded Negativity Conjecture~\ref{conj:BNG} holds for $X$.

\end{prop}
\begin{proof} Let $D$ be an irreducible curve such that $D^2<0$. Then $h^0(X,\mathcal O_X(D))=1$. If Conjecture~\ref{conj:H} holds, then 
\(h^1(X,\mathcal O_X(D)) \leq \alpha$, with $\alpha$ a constant. By Riemann-Roch one has
\[
1+h^2(X,\mathcal O_X(D))=p_a(X)+D^ 2-p_a(D)+2+ h^1(X,\mathcal O_X(D)),
\]
which implies
\[
D^ 2\geq -p_a(X)-\alpha-1,
\]
proving the Bounded Negativity Conjecture.\end{proof}

If $D^ 2\ge 0$, Conjecture~\ref{conj:H} is related to Question~\ref{q4}.

\begin{prop}
Let $X$ be a surface. If Conjecture~\ref{conj:H} holds for irreducible curves $D$ on $X$ such that $D^ 2\ge 0$, then 
Question~\ref{q4} has an affirmative answer for $X$ for irreducible curves $D$.\end{prop}
\begin{proof} Assume Conjecture~\ref{conj:H} holds for $X$, and we may assume $\alpha>0$. Let $D$ be an irreducible curve on $X$. 
By Riemann-Roch one has
\begin{align*}
h^0(X,\mathcal O_X(D))&\leq h^0(X,\mathcal O_X(D))+h^2(X,\mathcal O_X(D))\\
&=p_a(X)+1+\frac {D^ 2-K_X\cdot D}2+ h^1(X,\mathcal O_X(D))\\
&\leq	 p_a(X)+1+\frac {D^ 2-K_X\cdot D}2+ \alpha h^0(X,\mathcal O_X(D)).\\
\end{align*}
Since $D^ 2+2\geq h^0(X,\mathcal O_X(D))$, we have
\[
K_X\cdot D\leq 2(\alpha-1)(D^ 2+2)+2(p_a(X)+1)+D^ 2\leq (6\alpha+2p_a(X)-4)D^ 2,
\]
whence the assertion.\end{proof}

As well as Question~\ref{q4}, also  Conjecture~\ref{conj:H} is false in general. Indeed, 
Koll\'ar's Example~\ref{ex:K} provides a counterexample  to this conjecture too (see~\cite[Corollary~3.1.2]{AAVV}). 
In this setting it has been asked whether Conjecture~\ref{conj:H}  holds for rational surfaces.  This is probably false too, and in  a very strong sense. In fact we have:

\begin{prop} \label{prop:H} Assume that the SHGH conjecture holds. Then there is a rational surface $X$ and a sequence of irreducible curves 
\(\{G_n\}\) on $X$ with 
\[h^0(X,\cO_X(G_n)) = 1 \; \text{ and } \; h^1(X,\cO_X(G_n)) \;  \text{ arbitrarily large}.\]
\end{prop}

\begin{proof}
On \(X_{10}\), let \(\{D_n\}\) be the sequence of Pell divisors as in the counterexample to Questions~\ref{q2} and~\ref{q3}.  
Let \(\pi \colon X \to X_{10}\) be the double cover as above.  Finally, let \(\{G_n\} = \{\pi^\ast (D_n)\}\).  Then 
 \[
 \vdim(|G_n|)=- D_n\cdot H \quad \text{and} \quad h^0(X,\cO_X(G_n)) = 1\]
hence \[ \lim_n h^1(X,\cO_X(G_n))=+\infty , \]
as required.
\end{proof}

The surface \(X\) in the proof of Proposition~\ref{prop:H} is isomorphic to a blow-up of \(\P^2\) at \(21\) special points.  A natural question is the following.

\begin{question} For which surfaces does Harbourne's Conjecture~\ref{conj:H} hold?
  \end{question}

\end{document}